\documentclass[final,12]{elsarticle}

\usepackage[colorlinks,linkcolor=blue,citecolor=blue]{hyperref}
\usepackage{bookmark}
\usepackage{bm}

\usepackage{amsmath,amsfonts,amsmath,amssymb,amsthm,cases}

\usepackage{algorithm,algorithmic}
\usepackage{multirow}
\usepackage{longtable}
\usepackage{fullpage}
\usepackage{cleveref}
\usepackage{graphics,graphicx,epsfig,epstopdf,subfigure}
\usepackage{color,natbib}
\usepackage[dvipsnames, svgnames, x11names]{xcolor}
\usepackage{mathtools}
\usepackage{booktabs}

\newtheorem{theorem}{Theorem}[section]
\newtheorem{lemma}{Lemma}[section]

\newtheorem{example}{Example}[section]
\numberwithin{equation}{section}

\def \d {\mathrm{d}}
\newcommand{\bb}{\boldsymbol}

\begin{document}

\begin{frontmatter}
	\title{Optimal error analysis of an interior penalty virtual element method for fourth-order singular perturbation problems}
		
    \author[1]{Fang Feng}
	\ead{ffeng@njust.edu.cn}
    \author[1]{Yuanyi Sun}
    \address[1]{School of Mathematics and Statistics, Nanjing University of Science and Technology, Nanjing, Jiangsu 210094, China}

	\author[2,3,4]{Yue Yu}
	\ead{terenceyuyue@xtu.edu.cn}
\address[2]{School of Mathematics and Computational Science, Xiangtan University, Xiangtan, Hunan 411105, China}
\address[3]{Hunan Research Center of the Basic Discipline Fundamental Algorithmic Theory and Novel Computational Methods, Xiangtan, Hunan 411105, China}
\address[4]{National Center for Applied Mathematics in Hunan, Xiangtan, Hunan 411105, China}

\footnotetext{The authors are listed in alphabetical order. All authors contributed equally to this paper.}
	
\begin{abstract}
In recent studies \cite{ZZ24, FY24}, the Interior Penalty Virtual Element Method (IPVEM) has been developed for solving a fourth-order singular perturbation problem, with uniform convergence established in the lowest-order case concerning the perturbation parameter. However, the resulting uniform convergence rate is only of half-order, which is suboptimal. In this work, we demonstrate that the proposed IPVEM in fact achieves optimal and uniform error estimates, even in the presence of boundary layers. The theoretical results are substantiated through serveral numerical experiments, which confirm the validity of the error estimates and highlight the method's effectiveness for singularly perturbed problems.
\end{abstract}


\end{frontmatter}

\section{Introduction}\label{Sec1}

Let $\Omega$ be a bounded  polygonal domain of $\mathbb{R}^2$ with boundary $\partial \Omega$.  For $f\in L^2(\Omega)$, we consider the following boundary value problem of the fourth-order singular perturbation equation:
	\begin{equation}\label{strongform}
	\begin{cases}
			\varepsilon^2\Delta^2u-\Delta u = f & {\rm in}~~\Omega,\\
			u = \dfrac{\partial u}{\partial \bb n} =0& {\rm on}~~\partial \Omega,
	\end{cases}
	\end{equation}
where $\bm{n} = (n_1,n_2)$ is the unit outer normal to $\partial \Omega$, and $\varepsilon$ is a real parameter satisfying $0< \varepsilon\le 1$. In this article, we are primarily concerned with the case where $\varepsilon \to 0$, or the differential equation formally degenerates to the Poisson equation, but with the ``extra'' normal boundary condition, which will produce boundary layers in the first derivatives of $u$. 

{
This paper focuses on the Virtual Element Method (VEM), which extends the standard finite element method to accommodate general polytopal meshes \cite{Beirao-Brezzi-Cangiani-2013, Ahmad-Alsaedi-Brezzi-2013, Beirao-Brezzi-Marini-2014}. Compared to traditional finite element methods, VEMs offer advantages in handling complex geometries and problems with high-regularity solutions \cite{Brezzi-Marini-2013,Zhao-Chen-2016,Zhao-Zhang-Chen-2018,Chen-HuangX-2020}. Moreover, thanks to the great flexibility in mesh choice, VEMs require no special treatment for hanging nodes, which provides a significant advantage in mesh refinement \cite{ChenHuangLin2022avem,Carsten2023Morley}.}
Recently, the interior penalty technique for VEMs has been explored in \cite{ZMZW2023IPVEM} for the biharmonic equation, using the same degrees of freedom (DoFs) as for the $ H^1 $-conforming virtual elements \cite{Beirao-Brezzi-Cangiani-2013}. This numerical scheme can be viewed as a combination of the virtual element space and a discontinuous Galerkin method, as the resulting global discrete space is { $H^1$-nonconforming} and an interior penalty formulation is adopted to enforce the continuity of the solution.
The interior penalty method is then applied to tackle the fourth-order singular perturbation problem in \cite{ZZ24}, with certain adaptations to the original space. These modifications include changes in the definition of the $ H^1 $-type projection and the selection of the DoFs. In contrast to the approach in \cite{ZZ24}, we in \cite{FY24} adopt the original formulation introduced in \cite{ZMZW2023IPVEM} to address the fourth-order singular perturbation problem \eqref{strongform}. To streamline the implementation, we modify the jumps and averages in the penalty term by incorporating the elliptic projector $ \Pi_h^\nabla $. { It is worth noting that although the IPVEM is $H^1$-nonconforming and { thus requires lower regularity assumptions on the solution \cite{ZJZZ24} compared to $C^0$-continuous finite elements}, the a priori error estimates in \cite{ZMZW2023IPVEM,FY24,ZZ24} show that the interior penalty virtual elements behave similarly to $C^0$-continuous elements. This property can be extended to the a posteriori setting, as neither the discrete scheme nor the error estimators require explicit accounting of function jumps \cite{FHYZ2026}.}

Both analyses in \cite{ZZ24, FY24} demonstrate that the IPVEM is robust with respect to the perturbation parameter in the lowest-order case. { Although two methods achieve same uniform convergence rate of $\mathcal{O}(h^{1/2})$ in theory, their numerical performance differs due to the different handling of the consistency term
$\sum_{e \in \mathcal{E}_h} \int_e \frac{\partial u}{\partial \bb{n}_e}[v_h] \d s$. As noted in the introduction of \cite{FY24} and the remark below Theorem 4.5 of \cite{ZZ24}, the nonconformity error is not optimal for the method in \cite{ZZ24}, leading to a suboptimal convergence rate as $\varepsilon$ tends to zero. This is observed in Table 4 of that paper, where a convergence rate of $\mathcal{O}(h^{1.46})$ is reported, which is better than the theoretical prediction $\mathcal{O}(h^{1/2})$ but still falls short of the optimal rate $\mathcal{O}(h^2)$ for small $\varepsilon$.
In contrast, thanks to the vanishing consistency term, our method proposed in \cite{FY24}~--~despite also having a suboptimal error analysis~--~already exhibits optimal rates in numerical experiments. The present paper essentially closes the gap between theory and previously observed numerical results.}

{ Among the discrete methods that achieve optimal and uniform convergence for singularly perturbed problems, Nitsche's technique is commonly employed for parameter-dependent boundary value problems to enhance performance. For an overview of its applications, we refer to the review by \cite{CHH25}. In \cite{zbMATH07958861,Liu2024}, the authors show that the weak Galerkin finite element and discontinuous Galerkin methods can achieve optimal convergence by employing specially designed meshes. Since Nitsche's method requires additional stabilization terms, which may complicate implementation, \cite{HuangX2025} reformulates the problem as a second-order system, which allows the boundary condition $\partial_{\bm{n}} u$ to be naturally and weakly imposed in the distributional mixed formulation. Meanwhile, \cite{CHuangX2025} decouples the fourth-order elliptic singular perturbation problem into a combination of two Poisson equations and a generalized singularly perturbed Stokes-type equation involving the curl operator. To address the boundary layer, they introduce the nodal interpolation operator associated with the lowest-order N\'ed\'elec element of the second kind into the discrete bilinear forms, which similarly avoids additional stabilization terms.}

Upon reviewing this approach, we find that the IPVEM can be interpreted as an application of Nitsche's technique. In Nitsche's method, the boundary condition $\partial_{\bm{n}} u$ is imposed weakly via a penalty term, thereby imposing fewer restrictions on the boundary of the domain \cite{CHH25}. The IPVEM follows the same line, which becomes more apparent when examining the derivation of the scheme for inhomogeneous boundary conditions in \cite{FHYZ2026}. Consequently, the IPVEM should yield the optimal convergence rate, as previously observed in numerical tests. Building on this insight, we demonstrate in this article that the IPVEM developed in \cite{FY24} naturally achieves optimal convergence for fourth-order singularly perturbed problems with boundary layers.


\section{The continuous variational problem} \label{sec:cvariationalProb}

We first introduce some notations. For a bounded Lipschitz domain $D$ of dimension $d ~(d=1,2)$, the symbol $( v, w)_D = \int_D v w \d x$ denotes the $L^2$-inner product on $D$, $\|\cdot\|_{0,D}$ or $\|\cdot\|_{D}$ denotes the $L^2$-norm. If $D=\Omega$, we abbreviate $\|\cdot\|_{\Omega}$ as $\|\cdot\|$. {For vectorial functions,} $\bm{v} = (v_1,v_2)^T$ and $\bm{w} = (w_1,w_2)^T$, the inner product is defined as $( \bm{v}, \bm{w})_D = \int_D (v_1w_1 + v_2 w_2) \d x$.
 The set of scaled monomials $\mathbb{M}_r(D)$ is given by $\mathbb  M_{r} (D):=  \{  ( \frac{\boldsymbol x -  \boldsymbol x_D}{h_D} )^{\boldsymbol  s}, ~~ |\boldsymbol  s|\le r \}$, with the generic monomial denoted by $m_{D}$, where $h_D$ is the diameter of $D$, $\bm{x}_D$ is the centroid of $D$, and $r$ is a non-negative integer. For the multi-index ${\boldsymbol{s}} \in {\mathbb{N}^d}$, we follow the usual notation $\boldsymbol{x}^{\boldsymbol{s}} = x_1^{s_1} \cdots x_d^{s_d}$ with $|\bb{s}| = s_1 +  \cdots  + s_d$.
Conventionally, $\mathbb  M_r (D) =\{0\}$ for $r\le -1$.
Let $\mathcal{T}_h$ denote a tessellation of $\Omega$ into polygons. For a generic element $K$, define $h=\max_{K\in \mathcal{T}_h}h_K$ and $h_K={\rm diam}(K)$.
The set $\mathcal{E}_h$ represents all the edges in $\mathcal{T}_h$.
Let $e \subset \partial K$ be the common edge for elements { $K = K_-$ and $K_+$}, and let $v$ be a scalar function defined on $e$. We introduce the jump and average of $v$ on $e$ by $[v] = v^- - v^+$ and $\{v\} = \frac12 (v^- + v^+)$, where $v^-$ and $v^+$ are the traces of $v$ on $e$ from the interior and exterior of $K$, respectively. On a boundary edge, $[ v ] = v$ and $\{ v \} = v $. For a non-negative integer $m$ and an element $K\in \mathcal{T}_h$, denote by $\mathbb{P}_m(K)$ the set of all polynomials on $K$ with the total degree no more than $m$.
The broken Sobolev space is denoted as
\[H^m(\mathcal{T}_h):=\{w\in L^2(\Omega): w|_K\in H^m(K),  K\in \mathcal{T}_h \}\]
with the broken $H^m$-norm
\[
\|w\|_{m,h}:= (\sum_{K\in \mathcal{T}_h}\|w\|_{m,K}^2)^{1/2},\] and the broken $H^m$-seminorm
\[|w|_{m,h}=(\sum_{K\in \mathcal{T}_h}|v|_{m,K}^2)^{1/2}.\]
Moreover, for any two quantities $a$ and $b$, ``$a\lesssim b$" indicates ``$a\le C b$" with the constant $C$ independent of the mesh size $h_K$ {and $\varepsilon$}, and ``$a\eqsim b$" abbreviates ``$a\lesssim b\lesssim a$". The variational formulation of \eqref{strongform} is to find $u\in V:=H_0^2(\Omega)$ such that
\begin{equation}\label{weakform}
	\varepsilon^2a(u,v)+b(u,v)=(f,v),\quad  v\in H_0^2(\Omega),
\end{equation}
where
$a(u,v) = (\nabla^2u,\nabla^2v)$ and  $ b(u,v)=(\nabla u,\nabla v) $.

For the geometric subdivisions, we make the following assumption (cf. \cite{Brezzi-Buffa-Lipnikov-2009,Chen-HuangJ-2018}):
\begin{enumerate}[{\bf A1}.]
\item For each $K\in {\mathcal T}_h$, there exists a ``virtual triangulation" ${\mathcal T}_K$ of $K$ such that ${\mathcal T}_K$ is uniformly shape regular and quasi-uniform. The corresponding mesh size of ${\mathcal T}_K$ is proportional to $h_K$. Each edge of $K$ is a side of a certain triangle in ${\mathcal T}_K$.
\end{enumerate}
Based on this usual geometric {assumption},
by the standard Dupont-Scott theory \cite{BS2008}, for all $v\in H^l(K)$, there exists a certain $q \in \mathbb{P}_{l-1}(K)$ such that
\begin{equation}\label{BHe1}
|v - q|_{m,K} \lesssim h_K^{l - m} | v |_{l,K},\quad  0\le m\le l.
\end{equation}

\section{The interior penalty virtual element method} \label{sec:IPVEM}

The interior penalty virtual element method (IPVEM) was first proposed in \cite{ZMZW2023IPVEM}. In the construction, the authors first introduce a $C^1$-continuous virtual element space ($k \ge 2$)
\[\widetilde{V}_{k+2}(K)=\left\{v \in H^2(K):  \Delta^2 v \in \mathbb{P}_k(K), v|_e \in \mathbb{P}_{k+2}(e), \partial_{\bm n} v|_e \in \mathbb{P}_{k+1}(e), ~~e \subset \partial K\right\}.\]
For the equipped degree of freedoms (DoFs), we refer the reader to \cite{ZMZW2023IPVEM,FY24}:
{\begin{itemize}
	\item  $\tilde{\bm\chi}^p:$ the values of $v$ at the vertices of $K$,
   \begin{equation*}\label{dof1}
   \tilde{\chi}_z^p(v) = v(z), \quad \mbox{$z$ is a vertex of $K$}.
   \end{equation*}

	\item  $\tilde{ \bm \chi}^g$ : the values of $h_z \nabla v$ at the vertices  of $K$,
  \begin{equation*}\label{dof2}
  \tilde{\chi}_z^g(v) = h_z \nabla v(z), \quad \mbox{$z$ is a vertex of $K$},
   \end{equation*}
  where $h_z$ is a characteristic length attached to each vertex $z$.
	\item  $\tilde{\bm \chi}^e$ : the moments of $v$ on edges up to degree $k-2$,
   \begin{equation*}\label{dof3}
   \tilde{\chi}_e(v) = |e|^{-1}(m_e, v)_e, \quad m_e \in \mathbb{M}_{k-2}(e), \quad e \subset\partial K.
    \end{equation*}
	\item  $\tilde{\bm \chi}^n$ : the moments of $\partial_{\bm n_e} v$ on edges up to degree $k-1$,
  \begin{equation*}\label{dof4}
  { \tilde{\chi}_e^n(v)} = (m_e,\partial_{\bm n_e} v)_e, \quad  m_e \in \mathbb{M}_{k-1}(e), \quad e \subset \partial K.
   \end{equation*}
	\item  $\tilde{\bm \chi}^K$ : the moments on element $K$ up to degree $k$,
  \begin{equation*}\label{dof5}
  \tilde{\chi}_K(v) = |K|^{-1}(m_K,v)_K, \quad m_K \in \mathbb{M}_k(K).
  \end{equation*}
\end{itemize}}
Given $v_h \in \widetilde{V}_{k+2}(K)$, Ref.~\cite{ZMZW2023IPVEM} proposes an approximate $H^1$-elliptic projection $\Pi_K^{\nabla} v_h \in \mathbb{P}_k(K)$ described by the following equations:
\[
\left\{\begin{aligned}
(\nabla \Pi_K^{\nabla} v_h, \nabla q)_K & = -(v_h, \Delta q)_K + \sum\limits_{e\subset \partial K} Q_{2k-1}^e( v_h \partial_{\bm n_e} q), \quad q \in \mathbb{P}_k(K), \notag\\
	\sum_{z \in \mathcal{V}_K} \Pi_K^{\nabla} v_h(z)& =\sum_{z \in \mathcal{V}_K} v_h(z),\label{MH1}
\end{aligned}\right.
\]
where $\mathcal{V}_K$ is the set of the vertices of $K$, $Q_{2k-1}^e v := |e| \sum\limits_{i=0}^k \omega_i v(\bm{x}_i^e) \approx \int_e v(s) \d s$, with $(\omega_i, \bm{x}_i^e)$ being the $(k+1)$ Gauss-Lobatto quadrature weights and points, $\bm{x}_0^e$ and $\bm{x}_k^e$ are the endpoints of $e$. In this case, $\Pi_K^\nabla v_h$ is computable by only using the DoFs of the $H^1$-conforming virtual element space given by
\[V_k^{1,c}(K):= \left\{ v \in H^1(K): \Delta v \in \mathbb{P}_{k-2}(K),~~ v|_{\partial K} \in \mathbb{B}_k(\partial K) \right\},\]
and
\[
\mathbb{B}_k(\partial K) = \{ v \in C^0(\partial K): v|_e \in \mathbb{P}_k(e),~~ e \subset \partial K \}.\]

By using the standard enhancement technique \cite{Ahmad-Alsaedi-Brezzi-2013}, which substitutes the redundant DoFs of $v$ with those of $\Pi_K^{\nabla}v$, the local interior penalty virtual element space in \cite{ZMZW2023IPVEM} is defined as
\begin{align*}
V_k(K)
& = \{v \in \widetilde{V}_{k+2}(K): (v,m_K)_K = (\Pi_K^{\nabla}v,m_K)_K, ~ m_K\in \mathbb{M}_k(K) \backslash \mathbb{M}_{k-2}(K),   \nabla v(z)=\nabla \Pi_K^{\nabla}(z),\\
& ~\mbox{$z \in \mathcal{V}_K$ is any vertex of $K$},  (m_e,\partial_{\bm n_e} v)_e=(m_e,\partial_{\bm n_e} \Pi_K^{\nabla}v)_e,~ m_e \in \mathbb{M}_{k-1}(e), \, e \subset \partial K\}.
\end{align*}
The associated DoFs are given by
\begin{itemize}
    \item the values of $v(z)$, $z \in \mathcal{V}_K$,
    \item the values of $v(\boldsymbol{x}_i^e)$, $i=1, \cdots, k-1$, $e \subset \partial K$,
    \item the moments $|K|^{-1}( m_K, v)_K$, $m_K \in \mathbb{M}_{k-2}(K)$.
\end{itemize}
Furthermore, we use $V_h$ to denote the global nonconforming virtual element space, which is defined elementwise and required that the DoFs are continuous through the interior vertices and edges while zero for the boundary DoFs.

As usual, we can define the $H^2$-projection operator $\Pi_K^{\Delta}: V_k(K) \to \mathbb{P}_k(K)$ by finding the solution $\Pi_K^{\Delta} v \in \mathbb{P}_k(K)$ of
\[\begin{cases}
	a^K (\Pi_K^{\Delta} v, q )=a^K(v, q),\quad  q \in \mathbb{P}_k(K), \\
	\widehat{\Pi_K^{\Delta} v}=\widehat{v}, \quad \widehat{\nabla \Pi_K^{\Delta} v}=\widehat{\nabla v}
\end{cases}\]
for any given $v \in V_k(K)$, { where $a^K(\cdot, \cdot) = (\nabla^2 \cdot ,\nabla^2 \cdot )_K$} and
 the quasi-average $\widehat{v}$ is defined by
\[
\widehat{v}=\frac{1}{|\partial K|} \int_{\partial K} v \d s .
\]

\begin{lemma}\cite{ZMZW2023IPVEM,FY24} \label{lem:Pih}
Let $m=1,2$ and $0\le \ell \le m$. For all $v_h\in V_k(K)$, there hold
	\begin{align}
		&|v_h|_{m,K}\lesssim h_K^{\ell-m}|v_h|_{\ell,K} ,\label{inv}\\
		&|\Pi_K^{\nabla}v_h|_{m,K}\lesssim  |v_h|_{m,K},\label{Pi}\\
		&|v_h-\Pi_K^{\nabla}v_h|_{\ell,K}\lesssim  h_K^{\ell-m}|v_h|_{m,K}.\label{Proj}
	\end{align}
\end{lemma}


For the construction of the IPVEM, we do not use the form presented in \cite{ZMZW2023IPVEM}, but instead adopt the modified version from \cite{FY24}, which simplifies the implementation process.
For the fourth-order term, the discrete bilinear form is defined by
\[
a_h^K(v, w)=a^K(\Pi_K^{\Delta} v, \Pi_K^{\Delta} w)+S^K(v-\Pi_K^{\Delta} v, w-\Pi_K^{\Delta} w), \quad v, w \in V_k(K),
\]
with
\[S^K(v-\Pi_K^{\Delta} v,w-\Pi_K^{\Delta} w)=h_K^{-2} \sum\limits_{i=1}^{n_K} \chi_i(v-\Pi_K^{\Delta} v) \chi_i(w-\Pi_K^{\Delta} w),
\]
 where $\{\chi_i\}$ are the local DoFs on $K$ with $n_K$ being the number of the DoFs. For the second-order term, we define the local bilinear form as
\[b_h^K(v,w):=(\nabla \Pi_K^{\nabla}v,\nabla \Pi_K^{\nabla}w)_K + \sum_{i=1}^{n_K} \chi_i(v-\Pi_K^{\nabla} v) \chi_i (w-\Pi_K^{\nabla} w).\]
The interior penalty virtual element methods for the problem \eqref{strongform} can be described as follows: Find $u_h\in  V_h$ such that
\begin{equation}\label{IPVEM1}
	\varepsilon^2\mathcal{A}_h(u_h,v_h)+b_h(u_h, v_h) = \langle f_h,v_h\rangle,\quad v_h\in V_h,
\end{equation}
where
\[\mathcal{A}_h(u_h,v_h) = a_h(u_h,v_h)+J_1(u_h,v_h)+J_2(u_h,v_h)+J_3(u_h,v_h),\]
 with
 \[a_h(v,w)=\sum_{K\in \mathcal{T}_h} a_h^K(v,w),\quad\,b_h(v,w)=\sum_{K\in \mathcal{T}_h} b_h^K(v,w),\]
 and
\begin{align*}
& J_1(v,w)= \sum_{e\in\mathcal{E}_h}\frac{\lambda_e}{|e|}\int_e\Big[\frac{\partial \Pi_h^\nabla v}{\partial \bm{n}_e}\Big]\Big[\frac{\partial \Pi_h^\nabla w}{\partial \bm{n}_e}\Big] \d s \qquad \mbox{($\lambda_e\eqsim 1$ is some edge-dependent parameter \cite{FY24})},   \\
& J_2(v,w)=-\sum_{e \in \mathcal{E}_h}\int_e\Big\{\frac{\partial^2\Pi_h^{\nabla}v}{\partial \bm{n}_e^2}\Big\}\Big[\frac{\partial \Pi_h^{\nabla} w}{\partial \bm{n}_e}\Big] \d s, \\
&J_3(v,w)=-\sum_{e\in\mathcal{E}_h}\int_e\Big\{\frac{\partial^2 \Pi_h^{\nabla}w}{\partial \bm{n}_e^2}\Big\}\Big[\frac{\partial \Pi_h^{\nabla} v}{\partial \bm{n}_e}\Big] \d s.
\end{align*}

The right-hand side is $\langle f_h,v_h \rangle= \sum\limits_{K\in \mathcal{T}_h}\int_{K} f\,\Pi_{0,K}^k v_h \d x$, where $\Pi_{0,K}^k$ is the $L^2$ projector onto $\mathbb{P}_k(K)$. In what follows, we define
\[
\|w\|_h^2=|w|_{2,h}^2+J_1(w,w), \quad \| w \|_{\varepsilon,h} ^2=\varepsilon^2\|w\|_h^2+|w|_{1,h}^2.\]

\section{The optimal and uniform error estimate} \label{sec:err}

We first present some results for the error analysis.  Since $V_h$ and $V_h^{1,c}$ share the same DoFs, as done in \cite{FY24}, we can introduce the interpolation from $V_h$ to $V_h^{1,c}$. For any given $v_h\in V_h$, let $I_h^c v_h$ be the nodal interpolant of $v_h$ in $V_h^{1,c}$.
According to Lemmas 4.2 and 4.3 in \cite{FY24}, we have the following abstract Strang-type lemma.
\begin{lemma}\label{lem:StrangIPVEM}
	Assume that $u\in H_0^2(\Omega)\cap H^{k+1}(\Omega)$ is the solution to \eqref{weakform} with $k\ge 2$. Then we have the error estimate
\begin{equation}\label{Strangerr}
\|u-u_h\|_{\varepsilon,h}\lesssim \|u-I_hu\|_{\varepsilon,h} + \| u-u_{\pi} \|_{\varepsilon,h}  + \|f-f_h\|_{V_h'} + E_h(u),
\end{equation}
where $I_h$ is the interpolation operator to $V_h$ and $u_\pi$ is a piecewise polynomial with degree $\le k$ that satisfies \eqref{BHe1},
	\[\|f-f_h\|_{V_h'} = \sup _{v_h \in V_h\backslash \{0\}} \frac{\langle f - f_h, v_h\rangle}{\| v_h \|_{\varepsilon,h}}\]
	and the consistency term
	\begin{equation}\label{consistencyterm}
		E_h(u) = \sup _{v_h \in V_h \backslash\{0\}} \frac{E_A(u,v_h) + E_J(u,v_h)}{\|v_h\|_{\varepsilon,h}},
	\end{equation}
where
\begin{align*}
 &E_A(u,v_h) = E_{A1}(u,v_h) + E_{A2}(u,v_h)+E_{A3}(u,v_h), \\
 &E_J(u,v_h) = \varepsilon^2(J_1(I_hu,v_h)+J_3(I_hu,v_h)),
\end{align*}
with
\begin{align*}
& E_{A1}(u,v_h) = \varepsilon^2\sum_{e \in \mathcal{E}_h}\int_e  \Big\{\frac{\partial^2 (u-\Pi_h ^\nabla I_h u)}{\partial \boldsymbol{n}_e^2}\Big\} \left[\frac{\partial  v_h  }{\partial \boldsymbol{n}_e}\right]  \d s,\\
& E_{A2}(u,v_h) = \varepsilon^2\sum_{e \in \mathcal{E}_h}\int_e \frac{\partial^2 u}{\partial \boldsymbol{n}_e \partial \boldsymbol{t}_e}\left[\frac{\partial v_h}{\partial \boldsymbol{t}_e}\right] \d s, \\
& E_{A3}(u,v_h) = -\varepsilon^2 \sum_{K\in \mathcal{T}_h}(\nabla \Delta u,\nabla v_h)_K +  \sum_{K\in \mathcal{T}_h}(\nabla u,  \nabla I_h^c v_h)_K - (f, v_h).
\end{align*}
Here, $\bm{t}_e$ is a unit tangential vector of an edge $e\in \mathcal{E}_h$.
\end{lemma}

The following trace inequalities are very useful for our forthcoming analysis \cite{BS2008,Brenner2003}.
\begin{lemma}\label{lem:trace}
For any $K\in \mathcal{T}_h$, there hold
\begin{align}
&\|v\|_{0,\partial K} \lesssim \|v\|_{0,K}^{1/2} \|v\|_{1,K}^{1/2}, \quad v \in H^1(K),  \label{tracemul}\\
&\| v \|_{0,\partial K} \lesssim  h_K^{1/2}| v |_{1,K} +h_K^{ - 1/2}\| v \|_{0,K},\quad v \in H^1(K). \label{trace}
\end{align}
\end{lemma}

As shown in Lemma 3.11 of \cite{ZMZW2023IPVEM}, we can derive the following interpolation error estimate.
\begin{lemma}\label{lem:interpIPVEM}
	For any $v\in H^{\ell}(K)$ with $2\le \ell\le k+1$, it holds
	\begin{equation} \label{interp}
		|v-I_hv|_{m,K}\lesssim h_K^{\ell-m}|v|_{\ell,K},\quad m=0,1,2.
	\end{equation}	
\end{lemma}

Let $u^0$ be the solution of the following boundary value problem:
\begin{equation}\label{bdp}
	\begin{cases}
		-\Delta u^0 = f & \quad {\rm in} \quad \Omega,\\
		u^0=0 &  \quad {\rm on} \quad \partial \Omega.
	\end{cases}
\end{equation}
The following regularity is well-known and can be found in \cite{NTW01} for instance.
\begin{lemma}\label{lem:regular}
	If $\Omega$ is a bounded convex polygonal domain and $f\in L^2(\Omega)$, then
	\[
	|u|_2+\varepsilon|u|_3\lesssim \varepsilon^{-1/2}\|f\|\qquad \text{and}\qquad |u-u^0|_1\lesssim \varepsilon^{1/2}\|f\|.
	\]
\end{lemma}

 Now we consider the optimal and uniform error estimate. According to Lemma \ref{lem:StrangIPVEM}, it suffices to estimate the right-hand side of \eqref{Strangerr}. 
 
 { We first examine the interpolation and projection errors. In our previous work \cite{FY24}, we followed the proof in \cite[Theorem 4.5]{Zhao-Chen-2016}, where the approximation errors are expressed in terms of the mesh size $h$ and yield $\mathcal{O}(h^{1/2})$ under the $H^2$-regularity of $u^0$. However, by measuring the errors with respect to $\varepsilon$ (which is relevant for the small $\varepsilon$ regime) and assuming higher regularity of $u^0$, we obtain optimal approximation errors.}

\begin{lemma} \label{lem:approximation}
Suppose that $\Omega$ is a bounded convex polygonal domain and that $u^0\in H^m(\Omega)$ with $2\leq m\leq  k+1$.  Under the condition of Lemma \ref{lem:StrangIPVEM}, there hold
\[\|u-I_hu\|_{\varepsilon,h} + \| u-u_{\pi} \|_{\varepsilon,h} \lesssim \varepsilon^{1/2}\|f\|+h^{m-1}|u^0|_m,\quad\|f-f_h\|_{V_h'} \lesssim h^k |f |_{k-1}.\]
\end{lemma}
 \begin{proof}
 For the first term $\|u-I_hu\|_{\varepsilon,h}$, we write it explicitly as
\begin{equation*}\label{th1}
	\|u-I_hu\|_{\varepsilon,h}^2=\varepsilon^2|u-I_hu|_{2,h}^2+\varepsilon^2J_1(u-I_hu,u-I_hu)+|u-I_hu|_{1,h}^2.
\end{equation*}
By Lemmas \ref{lem:interpIPVEM} and \ref{lem:regular}, we derive
 \begin{align*}
	\varepsilon^2|u-I_hu|_{2,h}^2\leq\varepsilon^2\sum_{K \in \mathcal{T}_h}|u-I_Ku|_{2,K}^2\lesssim \varepsilon^2 |u|_2^2 \lesssim \varepsilon\|f\|^2.
 \end{align*}
Then, apply the trace inequality, the boundedness of $\Pi_K^{\nabla}$ in Lemma \ref{lem:Pih}, the interpolation error estimate in \eqref{interp} and the regularity estimate in Lemma \ref{lem:regular} to get
\begin{align*}
	\varepsilon^2J_1(u-I_hu,u-I_hu)&=\varepsilon^2\sum_{e\in\mathcal{E}_h}\int_e\frac{\lambda_e}{|e|}\Big[\frac{\partial \Pi_h^{\nabla}(u-I_hu)}{\partial \bm n_e}\Big]\Big[\frac{\partial \Pi_h^{\nabla}(u-I_hu)}{\partial \bm n_e}\Big]\mathrm{d} s\notag\\
	&\leq \varepsilon^2\sum_{{K}\in\mathcal{T}_h}\Big(|\Pi_h^{\nabla}(u-I_hu)|_{2,K}^2+h_K^{-2}|\Pi_h^{\nabla}(u-I_hu)|_{1,K}^2  \Big)\\
    &\lesssim \varepsilon^2 |u|_2^2 \lesssim \varepsilon\|f\|^2.
\end{align*}	
From Lemma \ref{lem:regular} and \eqref{interp}, we have
\begin{align*}\label{th2}
	|u-I_hu|_{1,h}^2
& = \sum_{K \in \mathcal{T}_h}|u-I_Ku|_{1,K}^2
  = \sum_{K \in \mathcal{T}_h}|u-u_0 + u_0 - I_K u_0 + I_K u_0 - I_Ku|_{1,K}^2\\
& \lesssim\sum_{K \in \mathcal{T}_h}\Big(|u-u^0-I_K(u-u^0)|_{1,K}^2
	+|u^0-I_Ku^0|_{1,K}^2\Big)\notag\\
& \lesssim \sum_{K \in \mathcal{T}_h}\Big(|u-u^0|_{1,K}^2+h_K^{2m-2}|u^0|_{m,K}^2\Big)\\
 &\lesssim |u-u_0|_1^2+h^{2m-2}|u^0|_m^2\\
 &\lesssim \varepsilon\|f\|^2+h^{2m-2}|u^0|_m^2.
\end{align*}
{ As observed, we measure the approximation errors in terms of $\varepsilon$ and assume higher regularity of $u^0$.}
Combining the above inequalities, we derive
\begin{equation}\label{opt:inter}
	\|u-I_hu\|_{\varepsilon,h}\lesssim \varepsilon^{1/2}\|f\|+h^{m-1}|u^0|_m.
\end{equation}
Similarly, we can bound the second term as $\|u-u_{\pi}\|_{\varepsilon,h}\lesssim \varepsilon^{1/2}\|f\|+h^{m-1}|u^0|_m.$
For the third term, from
\begin{align}\label{rhs}
\langle f-f_h, v_h \rangle
& = \sum_{K \in \mathcal{T}_h} (f, v_h-\Pi_{0,K}^k v_h )_K\notag\\
&=\sum_{K \in \mathcal{T}_h} (f-\Pi_{0,K}^{k-2}f, v_h-\Pi_{0,K}^k v_h )_K \nonumber\\
&  \lesssim \sum_{K \in \mathcal{T}_h} h_K^{k-1} |f |_{k-1,K} h_K | v_h |_{1,K}\notag \\
&\lesssim h^k |f |_{k-1} \| v_h \|_{\varepsilon,h},
\end{align}
we get $\|f-f_h\|_{V_h'} \lesssim h^k |f |_{k-1}$. The proof is completed.
 \end{proof}


The remaining part is to estimate the consistency term $E_h(u)$.
By Lemma \ref{lem:StrangIPVEM}, we need to consider the estimates of $E_{A1}(u,v_h)$, $E_{A2}(u,v_h)$, $E_{A3}(u,v_h)$,
$\varepsilon^2J_1(I_hu,v_h)$ and $\varepsilon^2J_3(I_hu,v_h)$. { Following the same spirit as for the approximation errors, we examine the consistency error in terms of the perturbation parameter $\varepsilon$, which yields its optimal convergence when $\varepsilon \lesssim h$.}	

\begin{lemma}\label{lem:consistency}
Suppose that $\Omega$ is a bounded convex polygonal domain. Under the condition of Lemma \ref{lem:StrangIPVEM}, there holds
$ E_h(u) \lesssim \varepsilon^{1/2}\|f\|+{h^{k} |f|_{k-1}  |v_h|_{1,h}}$.
\end{lemma}
\begin{proof}
For $E_{A1}$, by the trace inequality, the inverse inequality of polynomials, the boundedness of $\Pi_h^{\nabla}$, \eqref{inv}, \eqref{interp} and \eqref{BHe1}, we have
\begin{align*}
& E_{A1}(u,v_h) \\
&=\varepsilon^2\sum_{e \in \mathcal{E}_h}\int_e  \Big\{\frac{\partial^2 (u-\Pi_h ^\nabla I_h u)}{\partial \boldsymbol{n}_e^2}\Big\} \left[\frac{\partial  v_h  }{\partial \boldsymbol{n}_e}\right]  \d s
 \lesssim \varepsilon^2\sum_{e \in \mathcal{E}_h}\Big\| \frac{\partial^2 (u-\Pi_h ^\nabla I_h u)}{\partial \boldsymbol{n}_e^2}\Big\|_{0,e}\Big\| \frac{\partial  v_h  }{\partial \boldsymbol{n}_e}\Big\|_{0,e} \\
& \lesssim  \varepsilon^2\sum_{e \in \mathcal{E}_h}\Big\| \frac{\partial^2 (u-\Pi_h ^\nabla I_h u)}{\partial \boldsymbol{n}_e^2}\Big\|_{0,e}
(h_K^{-1/2}|v_h|_{1,K} + h_K^{1/2} |v_h|_{2,K}) \\
& \lesssim \varepsilon^2\sum_{K \in \mathcal{T}_h}\Big(h_K^{1/2}|u-\Pi_h^{\nabla}I_hu|_{3,K}+h_K^{-1/2}|u-\Pi_h^{\nabla}I_hu|_{2,K} \Big){ (h_K^{-1/2}|v_h|_{1,K})}  \\
& \lesssim \varepsilon^2\sum_{K \in \mathcal{T}_h}\Big( |u-u_{\pi}|_{3,K}+|\Pi_h^{\nabla}(u_{\pi}-I_hu)|_{3,K}+h_K^{-1}|u-u_{\pi}|_{2,K}+h_K^{-1}|\Pi_h^{\nabla}(u_{\pi}-I_hu)|_{2,K}\Big)|v_h|_{1,K} \\
& \lesssim \varepsilon^2\sum_{K \in \mathcal{T}_h}\Big( |u-u_{\pi}|_{3,K}+h_K^{-1}|u_{\pi}-I_hu|_{2,K}+h_K^{-1}|u-u_{\pi}|_{2,K}+h_K^{-1}|u_{\pi}-I_hu|_{2,K}\Big)|v_h|_{1,K} \\
& \lesssim \varepsilon^2 \sum_{K \in \mathcal{T}_h}|u|_{3,K}|v_h|_{1,K}\lesssim \varepsilon^2|u|_3|v_h|_{1,h}\lesssim \varepsilon^{1/2}\|f\|   |v_h|_{1,h}.
\end{align*}	
For $E_{A2}$, let $w|_e:=\frac{\partial^2 u}{\partial \boldsymbol{n}_e \partial \boldsymbol{t}_e}$. Using the trace inequality, the inverse inequality, the projection error estimate and Lemma \ref{lem:regular}, we get
\begin{align*}
E_{A2}(u,v_h)
& = \varepsilon^2\sum_{e \in \mathcal{E}_h}\int_e \frac{\partial^2 u}{\partial \boldsymbol{n}_e \partial \boldsymbol{t}_e}\left[\frac{\partial v_h}{\partial \boldsymbol{t}_e}\right] \d s  \lesssim \varepsilon^2\sum_{e \in \mathcal{E}_h}\int_e(w-\Pi_{0,K}^{k-2}w)\left[\frac{\partial v_h}{\partial \boldsymbol{t}_e}\right]\d s\\
&\lesssim {\varepsilon^2} \sum_{e \in \mathcal{E}_h}\|w-\Pi_{0,K}^{k-2} w \|_e\Big\| \left[\frac{\partial v_h}{\partial \boldsymbol{t}_e}\right]\Big\|_e\notag\\
& \lesssim \varepsilon^2\sum_{K \in \mathcal{T}_h}|u|_{3,K}|v_h|_{1,K}\lesssim \varepsilon^{1/2}\|f\||v_h|_{1,h}.
\end{align*}
For $E_{A3}$, the integration by parts in \eqref{weakform} gives
\[-\varepsilon^2 (\nabla (\Delta u), \nabla \varphi)+  (\nabla u, \nabla \varphi) = (f,\varphi),\quad  \varphi \in C_0^\infty(\Omega),\]
which implies
\begin{align*}
	-\varepsilon^2\sum_{K\in \mathcal{T}_h}(\nabla(\Delta u),\nabla I_h^c v_h )_K+\sum_{K\in \mathcal{T}_h}(\nabla u,\nabla I_h^c v_h )_K = (f,I_h^c v_h),
\end{align*}
since $C_0^\infty(\Omega)$ is dense in $H_0^1(\Omega)$, $u\in H^3(\Omega)$ and $I_h^c v_h \in H_0^1(\Omega)$. Then we can write $E_{A3}$ as
\begin{equation*}\label{eq:EA3}
	E_{A3}(u,v_h) =\sum_{K\in \mathcal{T}_h}\varepsilon^2\big(\nabla \Delta u,\nabla( I_h^c v_h-v_h)\big)_K+(f,I_h^c v_h-v_h).
\end{equation*}
According to Lemma \ref{lem:regular} and the error estimate for interpolation operator $I_h^c$ \cite{FY24}, one gets
\begin{align}\label{EA3}
	E_{A3}(u,v_h)
	& \lesssim \sum_{K\in \mathcal{T}_h} ( \varepsilon^2 |u|_{3,K} |I_h^c v_h-v_h|_{1,K} + h_K^{k-1} |f|_{k-1,K} \|I_h^c v_h-v_h\|_{0,K} )\notag\\
    &\lesssim\varepsilon^{1/2}\|f\||v_h|_{1,h}+h^{k} |f|_{k-1}  |v_h|_{1,h}.
\end{align}

For $\varepsilon^2J_1(I_hu,v_h)$, we apply the trace inequality, the boundedness of $\Pi_h^{\nabla}$, \eqref{interp} and \eqref{BHe1} to get
\begin{align*}
 \varepsilon^2 J_1(I_h u,v_h)
& = \varepsilon^2 \sum_{e\in\mathcal{E}_h}\frac{\lambda_e}{|e|}\int_e\Big[\frac{\partial \Pi_h^{\nabla}I_hu }{\partial \bm n_e}\Big]	\Big[\frac{\partial \Pi_h^{\nabla}v_h}{\partial \bm n_e}\Big] \mathrm{d} s	= \varepsilon^2 \sum_{e\in\mathcal{E}_h}\frac{\lambda_e}{|e|}\int_e\Big[\frac{\partial (\Pi_h^{\nabla}I_hu-u)}{\partial \bm n_e}\Big]\Big[\frac{\partial \Pi_h^{\nabla}v_h}{\partial \bm n_e}\Big] \mathrm{d} s\notag\\
&\le \varepsilon^2 \Big(\sum_{e\in\mathcal{E}_h} \frac{\lambda_e}{|e|} \Big\| \Big[\frac{\partial (\Pi_h^{\nabla}I_hu-u)}{\partial \bm n_e}\Big] \Big\|_{0,e}^2\Big)^{1/2}J_1(v_h,v_h)^{1/2}\notag\\
& \lesssim \varepsilon\Big(\sum_{K\in\mathcal{T}_h}|u - I_hu|_{2,K}^2 + |u-u_\pi|_{2,K}^2 + h_K^{-2}(|u-I_hu|_{1,K}^2 + |u-u_\pi|_{1,K}^2 )\Big)^{1/2}
\| v_h \|_{\varepsilon,h}\notag\\
& \lesssim \varepsilon |u|_2 {\| v_h \|}_{\varepsilon,h}\lesssim \varepsilon^{1/2}\|f\|\| v_h \|_{\varepsilon,h}.
\end{align*}

For $\varepsilon^2J_3(I_hu,v_h)$, the continuity of $u$, the trace inequality, the inverse inequality for polynomials, the boundedness of $\Pi_h^{\nabla}$ in \eqref{Pi} and the Cauchy-Schwarz inequality give
\begin{align*}
\varepsilon^2J_3(I_hu,v_h)
&  = \varepsilon^2\sum_{{e\in\mathcal{E}_h}}\int_e\Big[\frac{\partial (\Pi_h^{\nabla}I_hu-u)}{\partial \bm n_e}\Big]\Big[\frac{\partial^2\Pi_h^{\nabla}v_h}{\partial \bm n_e^2}\Big]\mathrm{d} s\notag\\
&\lesssim \varepsilon^2\sum_{{e\in \mathcal{E}_h}}\Big\|\Big[\frac{\partial( \Pi_h^{\nabla}I_hu-u)}{\partial \bm n_e}\Big] \Big\|_{0,e}\Big\|\Big[\frac{\partial^2\Pi_h^{\nabla}v_h}{\partial \bm n_e^2}\Big] \Big\|_{0,e}\notag\\
& \lesssim \varepsilon^2 \sum_{K\in\mathcal{T}_h}\Big(h_K^{1/2}|\Pi_h^{\nabla}I_hu-u|_{2,K} + h_K^{-1/2}|\Pi_h^{\nabla}I_hu-u|_{1,K}\Big) \Big(h_K^{-1/2}|\Pi_h^{\nabla}v_h|_{2,K}+h_K^{1/2}|\Pi_h^{\nabla}v_h|_{3,K}\Big)\notag\\
&\lesssim \varepsilon^2 \sum_{K\in \mathcal{T}_h}\Big( |u-I_hu|_{2,K}+|u-u_{\pi}|_{2,k}+h_K^{-1}(|u-I_hu|_{1,k}+|u-u_{\pi}|_{1,K})\Big){|v_h|_{2,K}}\notag\\
&\lesssim \varepsilon|u|_{2}\|v_h\|_{\varepsilon,h}
 \lesssim \varepsilon^{1/2}\|f\|\| v_h \|_{\varepsilon,h}.
\end{align*}
The proof is completed.
\end{proof}


In summary of the above estimates, we conclude that

\begin{theorem}\label{thm:unif}
Let $u\in H_0^2(\Omega)\cap H^{k+1}(\Omega)$ be the solution to \eqref{weakform} with $k \ge 2$.  Suppose that $\Omega$ is a bounded convex polygonal domain and that $u^0\in H^m(\Omega)$ with $2\leq m\leq  k+1$. Then, it follows
\begin{equation}\label{result}
\|u-u_h\|_{\varepsilon,h}\lesssim \varepsilon^{1/2}\|f\|+h^{m-1}|u^0|_m+h^k|f|_{k-1}.
\end{equation}
\end{theorem}

Compared to the result in \cite{FY24}, we obtain the second-order convergence rate in the lowest-order case ($k=2$) if $f\in H^{k-1}(\Omega)$, which is optimal, as $\varepsilon$ approaches zero.

\section{Numerical examples}

In this section, we report the performance with several examples by testing the accuracy and the robustness with respect to the singular parameter $\varepsilon$ to resolve solutions with boundary layer effects.
We have considered both the case of $ k = 2$ and the case of $k = 3$ and  the domain $\Omega$ is taken as the unit square $(0,1)^2$.

All examples are implemented in MATLAB R2019b. The code is available on GitHub (\url{https://github.com/Terenceyuyue/mVEM}) as part of the mVEM package, which includes efficient and easy-to-follow implementations for various VEMs from the literature.
The subroutine `Fourth\_order\_Singular\_Perturbation\_IPVEM.m' is used to compute the numerical solutions, while the test script `main\_Fourth\_order\_Singular\_Perturbation\_IPVEM.m' is used to verify the convergence rates.

\begin{example}\label{ex:1}
Let $\Omega=[0,1]\times [0,1]$.
The solution
with a layer is constructed such that { $u(x,y) =\sin^2(\pi x y)+ 3 x^2 y^2 $.}
\end{example}

Let $ u $ be the exact solution of \eqref{strongform}, and $ u_h $ the discrete solution obtained from the underlying VEM \eqref{IPVEM1}. Since the VEM solution is not explicitly available within the polygonal elements, we will assess the errors by comparing the exact solution $ u $ with the elliptic projections. In this manner, the discrete error, measured in terms of the discrete energy norm, is quantified by
\begin{equation}\label{errEval}
\text{Err} = \Big( \sum\limits_{K \in \mathcal{T}_h} (\varepsilon^2 |u -  \Pi_h^\Delta  u_h|_{2,K}^2 + |u - \Pi_h^\nabla u_h|_{1,K}^2 ) \Big)^{1/2}/\|f\|.
\end{equation}

To evaluate the accuracy of the proposed method, we analyze a series of meshes based on a Centroidal Voronoi Tessellation of the unit square, with polygonal counts of $N$= 32, 64, 128, 256, and 512. These meshes are generated using the MATLAB toolbox PolyMesher \cite{Talischi-Paulino-Pereira-2012}.
The convergence orders of the errors with respect to the mesh size $h$ for $k=2$ and { $k=3$} are presented in {Table \ref{tab:ex1} and Table \ref{tab:ex1_1}}, respectively.
{ It is evident that the VEM guarantees at least $k$-th order convergence for all $\varepsilon \in (0,1]$. As $\varepsilon$ approaches zero, a $(k+1)$-th order convergence rate is observed, which agrees with the theoretical prediction in Theorem \ref{thm:unif}. Furthermore, the error behavior remains stable in this regime.}

\begin{table}[!htb]
  \centering
  \caption{The convergence rate with $k=2$ for Example \ref{ex:1}}\label{tab:ex1}
  \begin{tabular}{ccccccccccccccccc}
  \toprule[0.2mm]
  $\varepsilon \backslash N$ & 32   &64   &128   &256   &512  & Rate\\
  \midrule[0.3mm]
  1e-0  & 1.6926e-02   & 1.1886e-02   & 7.7060e-03   & 6.0383e-03   & 3.8926e-03   & 1.04\\
  1e-1  & 6.0237e-02   & 3.8000e-02   & 2.4672e-02   & 1.9601e-02   & 1.3106e-02   & 1.07\\
  1e-2  & 1.4605e-02   & 8.4552e-03   & 3.7313e-03   & 2.3571e-03   & 1.4715e-03   & 1.69\\
  1e-3  & 1.2650e-02   & 6.8511e-03   & 2.7067e-03   & 1.3758e-03   & 6.9842e-04   & 2.13\\
  1e-4  & 1.2636e-02   & 6.8360e-03   & 2.6888e-03   & 1.3629e-03   & 6.8605e-04   & 2.15\\
  1e-5  & 1.2636e-02   & 6.8359e-03   & 2.6888e-03   & 1.3628e-03   & 6.8597e-04   & 2.15\\
  \bottomrule[0.2mm]
\end{tabular}
\end{table}

\begin{table}[!htb]
  \centering
  \caption{The convergence rate {with $k=3$} for Example \ref{ex:1}}\label{tab:ex1_1}
  \begin{tabular}{ccccccccccccccccc}
  \toprule[0.2mm]
  $\varepsilon \backslash N$ & 32   &64   &128   &256   &512  & Rate\\
  \midrule[0.3mm]
  1e-0   & 3.8262e-03   & 1.4450e-03   & 6.2458e-04   & 2.6427e-04   & 1.4084e-04   & 2.40\\
  1e-1   & 1.3452e-02   & 5.0626e-03   & 2.1361e-03   & 9.2470e-04   & 4.9631e-04   & 2.39 \\
  1e-2   & 3.0331e-03   & 1.2243e-03   & 3.7252e-04   & 1.5741e-04   & 7.5709e-05   & 2.72 \\
  1e-3   & 1.9226e-03   & 7.1015e-04   & 1.9430e-04   & 7.5048e-05   & 2.9407e-05   & 3.06 \\
  1e-4   & 1.9592e-03   & 7.1154e-04   & 1.8983e-04   & 6.8974e-05   & 2.5094e-05   & 3.19\\
  1e-5   & 1.9598e-03   & 7.1187e-04   & 1.8989e-04   & 6.8994e-05   & 2.5104e-05   & 3.19\\
  \bottomrule[0.2mm]
\end{tabular}
\end{table}

\begin{example}\label{ex:2}
In this example, we test the performance of the VEM to resolve a solution with
boundary layer effect. The solution is construct such that
\begin{align*}
u=\left(\exp (\sin \pi x)-1-\pi \varepsilon \frac{\cosh \frac{1}{2 \varepsilon}-\cosh \frac{2 x-1}{2 \varepsilon}}{\sinh \frac{1}{2 \varepsilon}}\right)
\times\left(\exp (\sin \pi y)-1-\pi \varepsilon \frac{\cosh \frac{1}{2 \varepsilon}-\cosh \frac{2 y-1}{2 \varepsilon}}{\sinh \frac{1}{2 \varepsilon}}\right).
\end{align*}
\end{example}

For this boundary layer problem, we focus on the behavior when $\varepsilon$ is small. {We evaluate
the performance on  a square domain}. Please refer to Fig. \ref{fig:example2} for a snapshot of the numerical and
exact solutions with $\varepsilon = 10^{-10}$. Additionally, we present the error and the corresponding convergence rate for $k=2$ and $k=3$  in Tables \ref{tab:ex2} and \ref{tab:ex2_1}, respectively. For sufficiently small $\varepsilon$, the observed convergence rates agree with the theoretical prediction in Theorem \ref{thm:unif}.

\begin{table}[!htb]
  \centering
  \caption{The convergence rate for Example \ref{ex:2} with $k=2$ on a square domain}\label{tab:ex2}
  \begin{tabular}{ccccccccccccccccc}
 \toprule[0.2mm]
  $\varepsilon \backslash N$ & 100   &200   &300   &400   &500  & Rate\\
  \midrule[0.3mm]
   1e-4  & 7.8922e-03   & 4.0052e-03   & 2.0022e-03   & 1.0141e-03   & 4.9884e-04   & 1.99\\
   1e-5 & 7.8921e-03   & 4.0052e-03   & 2.0022e-03   & 1.0140e-03   & 4.9877e-04   &1.99\\
   1e-6  & 7.8921e-03   & 4.0052e-03   & 2.0022e-03   & 1.0140e-03   & 4.9877e-04   & 1.99\\
   1e-7& 7.8921e-03   & 4.0052e-03   & 2.0022e-03   & 1.0140e-03   & 4.9877e-04   & 1.99\\

  \bottomrule[0.2mm]
\end{tabular}
\end{table}

\begin{table}[!htb]
  \centering
  \caption{The convergence rate for Example \ref{ex:2}with {$k=3$} on a square domain}\label{tab:ex2_1}
  \begin{tabular}{ccccccccccccccccc}
 \toprule[0.2mm]
  $\varepsilon \backslash N$ & 100   &200   &300   &400   &500  & Rate\\
  \midrule[0.3mm]
   1e-4  & 1.1235e-03   & 3.8181e-04   & 1.3050e-04   & 4.5905e-05   & 1.6040e-05   & 3.06\\
   1e-5 & 1.1238e-03   & 3.8195e-04   & 1.3056e-04   & 4.5928e-05   & 1.6044e-05   &3.06\\
   1e-6  & 1.1238e-03   & 3.8195e-04   & 1.3056e-04   & 4.5928e-05   & 1.6044e-05   & 3.06\\
   1e-7& 1.1238e-03   & 3.8195e-04   & 1.3056e-04   & 4.5928e-05   & 1.6044e-05   &3.06\\
  \bottomrule[0.2mm]
\end{tabular}
\end{table}

 \begin{figure}[!htb]
  \centering
  \includegraphics[scale=0.48,trim = 50 0 50 0,clip]{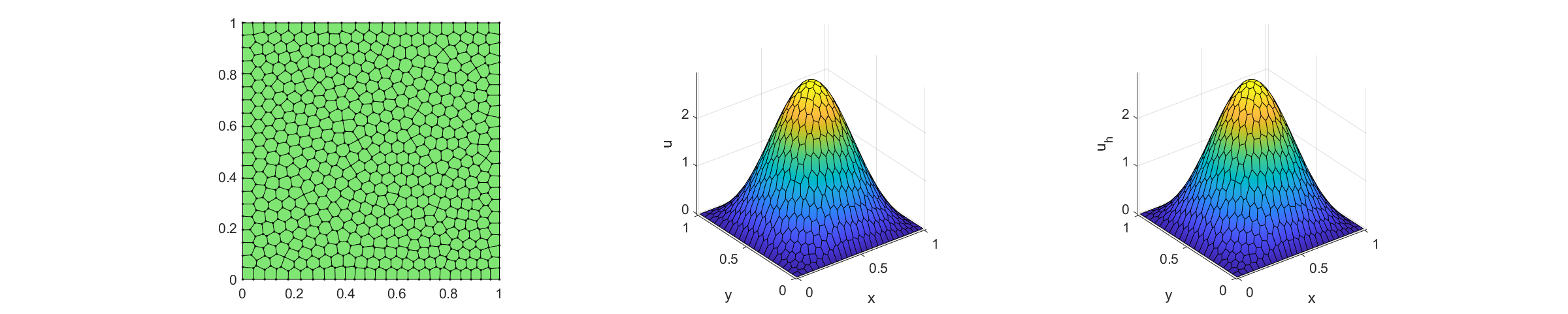}\\
  \caption{Numerical and exact solutions for Example \ref{ex:2} with $k=2$~($\varepsilon = 10^{-10}$) }\label{fig:example2}
\end{figure}

\section*{Acknowledgements}

FF was partially supported by the National Natural Science Foundation of China (NSFC) grant  12401528 and
the Fundamental Research Funds for the Central Universities, No. 30924010837.
YY was partially supported by NSFC grant (No.\ 12301561), the Key Project of Scientific Research Project of Hunan Provincial Department of Education (No.\ 24A0100), the Science and Technology Innovation Program of Hunan Province (No.\ 2025RC3150) and the general program of Hunan Provincial Natural Science Foundation (No.\ 2026JJ50003).
This research was supported in part by the 111 Project (No.\ D23017), and Program for Science and Technology Innovative Research Team in Higher Educational Institutions of Hunan Province of China.

\bibliographystyle{plain} 
\bibliography{Refs_IPVEM}

\end{document}